\documentclass[11pt,reqno]{amsart}
\usepackage{latexsym}
\usepackage{mathtools}
\usepackage{hyperref}
\usepackage{a4wide}
\usepackage{enumerate}
\newcommand{\R}{\ensuremath{\mathbb{R}}}

\newcommand{\T}{S}
\newcommand{\E}{\mathrm{e}}

\newcommand{\W}{h}
\newcommand{\WW}{\mathcal{W}}
\newcommand{\eL}{\mathcal{L}}

\newcommand{\pd}{\partial}
\newcommand{\cd}{\nabla}

\newcommand{\Ges}{G_{\varepsilon,\sigma}}
\newcommand{\lb}{\left(}
\newcommand{\lsb}{\left[}
\newcommand{\lcb}{\left\{}
\newcommand{\rb}{\right)}
\newcommand{\rsb}{\right]}
\newcommand{\rcb}{\right\}}
\theoremstyle{plain}
\numberwithin{equation}{section}
\newtheorem{thm}{Theorem}[section]

\newtheorem{lem}[thm]{Lemma}

\newtheorem{cor}[thm]{Corollary}

\newtheorem{prop}[thm]{Proposition}

\newtheorem*{conds*}{Conditions}
\newtheorem*{auxconds*}{Ancillary Conditions}
\newtheorem*{props*}{Properties}
\theoremstyle{remark}

\def\ba #1\ea {\begin{eqnarray} #1\end{eqnarray}}
\def\benn #1\eenn {\begin{eqnarray*} #1\end{eqnarray*}}
\def\ba #1\ea {\begin{align} #1\end{align}}
\def\bann #1\eann {\begin{align*} #1\end{align*}}
\def\ben #1\een {\begin{enumerate} #1\end{enumerate}}
\def\bi #1\ei {\begin{itemize} #1\end{itemize}}

\title[Cylindrical estimates]{Cylindrical estimates for hypersurfaces moving by convex curvature functions}

\author{Ben Andrews}
\address{Mathematical Sciences Institute, Australian National University, ACT 0200 Australia}
\address{Mathematical Sciences Center, Tsinghua University, Beijing 100084, China}
\address{Morningside Center for Mathematics, Chinese Academy of Sciences, Beijing 100190, China}
\email{ben.andrews@anu.edu.au}
\thanks{Research partially supported by Discovery grant DP120102462 of the Australian Research Council}
\author{Mat Langford}
\address{Mathematical Sciences Institute, Australian National University, ACT 0200 Australia}
\email{mathew.langford@anu.edu.au}
\thanks{The second author gratefully acknowledges the support of an Australian Postgraduate Award during the completion of this work.}

\begin{document}

\begin{abstract}
We prove a complete family of `cylindrical estimates' for solutions of a class of fully non-linear curvature flows, generalising the cylindrical estimate of Huisken-Sinestrari \cite[Section 5]{HuSi09} for the mean curvature flow. More precisely, we show that, for the class of flows considered, an $(m+1)$-convex ($0\leq m\leq n-2$) solution becomes either strictly $m$-convex, or its Weingarten map approaches that of a cylinder $\R^m\times S^{n-m}$ at points where the curvature is becoming large. This result complements the convexity estimate proved in \cite{hypersurfaces} for the same class of flows.
\end{abstract}

\maketitle

\section{Introduction}

Let $M$ be a smooth, closed manifold of dimension $n$, and $X_0:M\to\R^{n+1}$ a smooth hypersurface immersion. We are interested in smooth families $X:M\times [0,T)\to \R^{n+1}$ of smooth immersions $X(\cdot,t)$ solving the initial value problem
\begin{equation}\label{eq:CF}\tag{CF}
\begin{dcases}
\;\; \pd_t X(x,t)=-F(\WW(x,t))\nu(x,t)\,,\\
\;\; X(\cdot,0)=X_0\,.
\end{dcases}
\end{equation}
where $\nu$ is the outer normal field of the evolving hypersurface $X$ and $\WW$ the corresponding Weingarten curvature. In order that the problem \eqref{eq:CF} be well posed, we require that $F(\WW)$ be given by a smooth, symmetric, degree one homogeneous function $f:\Gamma\to\R$ of the principal curvatures $\kappa_i$ which is monotone increasing in each argument. The symmetry of $f$ ensures that $F$ is a smooth, basis invariant function of the components of the Weingarten map (or an orthonormal frame invariant function of the components of the second fundamental form) \cite{Gl}. Monotonicity implies monotonicity with respect to the Weingarten curvature, which ensures that the flow is (weakly) parabolic. This guarantees local existence of solutions of \eqref{eq:CF}, as long as the principal curvature $n$-tuple of the initial data lies in $\Gamma$ \cite[Main Theorem 5]{Baker}.

For technical reasons, we require the following additional conditions
\begin{conds*}
\mbox{}
\ben[(i)]
\item\label{cond:homogeneity} that $\Gamma$ is a convex cone\footnote{We remark that this condition can be slightly weakened. See \cite{hypersurfaces}.}, and $f$ is homogeneous of degree one; and
\item\label{cond:convexity} that $f$ is convex.
\newcounter{enumi_saved}
\setcounter{enumi_saved}{\value{enumi}}
\een
\end{conds*}
Then, since the normal points out of the region enclosed by the solution, we may assume that $(1,\dots,1)\in\Gamma$, and we we lose no generality in assuming that $f$ is normalised such that $f(1,\dots,1)=1$. 

The additional conditions (\ref{cond:homogeneity})-(\ref{cond:convexity}) have several consequences. Most importantly, they allow us to obtain a preserved cone $\Gamma_0\subset\Gamma$ of curvatures for the flow \cite[Lemma 2.4]{hypersurfaces}. This allows us to obtain uniform estimates on any degree zero homogeneous function of curvature along the flow (Lemma \ref{lem:homogeneousbounds}); in particular, we deduce uniform parabolicity of the flow (Corollary \ref{cor:uniformparabolicity}). The convexity condition then allows us to apply the second derivative H\"older estimate of Evans \cite{E} and Krylov \cite{K} to deduce that the solution exists on maximal time interval $[0,T)$, $T<\infty$, such that $\max_{M\times\{t\}}F\to\infty$ as $t\to T$, as in \cite[Proposition 2.6]{surfaces}. This paper addresses the behaviour of solutions as $F\to\infty$. Let us recall the following curvature estmate \cite{hypersurfaces} (cf. \cite{HuSi99a,HuSi99b}):
\begin{thm}[Convexity Estimate]\label{thm:AC}
Let $X:M\times[0,T)\to\R^{n+1}$ be a solution of \eqref{eq:CF} such that $f$ satisfies Conditions (\ref{cond:homogeneity})--(\ref{cond:convexity}). Then for all $\varepsilon>0$ there is a constant $C_\varepsilon<\infty$ such that
\bann
G(x,t) \leq{}&\varepsilon F(x,t) +C_\varepsilon\quad \mbox{for all}\quad (x,t)\in M\times [0,T)\,,
\eann
where $G$ is given by a smooth, non-negative, degree one homogeneous function of the principal curvatures of the evolving hypersurface that vanishes at a point $(x,t)$ if and only if $\WW_{(x,t)}\geq 0$.
\end{thm}
Theorem \ref{thm:AC} implies that the ratio of the smallest principal curvature to the speed is almost positive wherever the curvature is large. Combining it with the differential Harnack inequality of \cite{An94b} and the strong maximum principle \cite{Ha84} yields useful information about the geometry of solutions of \eqref{eq:CF} near singularities \cite{hypersurfaces} (cf. \cite{HuSi99a,HuSi99b}):
\begin{cor}\label{cor:AC}
Any blow-up limit of a solution of \eqref{eq:CF} is weakly convex. In particular, any type-II blow-up limit of a solution of \eqref{eq:CF} about a type-II singularity is a translation solution of \eqref{eq:CF} of the form $X_\infty:(\R^k\times\Gamma^{n-k})\times \R\to\R^{n+1}$, $k\in\{0,1,\dots,n-1\}$, such that $X_\infty|_{\Gamma^{n-k}}$ is a strictly convex translation solution of \eqref{eq:CF} in $\R^{n-k+1}$.
\end{cor}

Motivated by \cite[Section 5]{HuSi09}, we apply Theorem \ref{thm:AC} to obtain the following family of cylindrical estimates for solutions of \eqref{eq:CF}:

\begin{thm}[Cylindrical Estimate]\label{thm:CE}
Let $X$ be a solution of \eqref{eq:CF} such that Conditions (\ref{cond:homogeneity})--(\ref{cond:convexity}) hold. Suppose also that $X$ is $(m+1)$-convex for some $m\in\{0,1,\dots,n-2\}$. That is, $\kappa_1+\dots+\kappa_{m+1}\geq \beta F$ for some $\beta>0$. Then for all $\varepsilon>0$ there is a constant $C_\varepsilon>0$ such that
\bann
G_m(x,t) \leq{}&\varepsilon F(x,t) +C_\varepsilon\quad \mbox{for all}\quad (x,t)\in M\times [0,T)\,,
\eann
where $G_m:M\times [0,T)\to\R$ is given by a smooth, non-negative, degree one homogeneous function of the principal curvatures that vanishes at a point $(x,t)$ if and only if
\bann
\kappa_1(x,t)+\dots+\kappa_{m+1}(x,t)\geq {}&\frac{1}{c_m}f(\kappa_1(x,t),\dots,\kappa_n(x,t))\,,
\eann
where $c_m$ is the value $F$ takes on the unit radius cylinder, $\R^m\times \T^{n-m}$.
\end{thm}
Theorem \ref{thm:CE} implies that the ratio of the quantity 
$$
K_m:=\kappa_1+\dots+\kappa_{m+1}-\frac{1}{c_m}F
$$ 
to the speed is almost positive wherever the curvature is large. Observe that this quantity is non-negative on a weakly convex hypersurface $\Sigma$ only if either $\Sigma$ is strictly $m-$convex, or $\Sigma=\R^m\times S^{n-m}$. In particular, we find that whenever $\kappa_1(x,t)+\dots+\kappa_{m}(x,t)$ is small compared to the speed, the Weingarten curvature is close to that of a thin cylinder $\R^m\times S^{n-m}$. We obtain the following refinement of Corollary \ref{cor:AC}:

\begin{cor}\label{cor:corollary1}
Any blow-up limit of an $(m+1)$-convex, $0\leq m\leq n-2$, solution of \eqref{eq:CF} is either strictly $m$-convex, or a shrinking cylinder $\R^m\times S^{n-m}$. In particular, if the blow-up is of type-II, then this limit is a translation solution of \eqref{eq:CF} of the form $X_\infty:\lb\R^k\times\Gamma^{n-k}\rb\times \R\to\R^{n+1}$ for $k\in\{0,1,\dots,m-1\}$, such that $X_\infty|_{\Gamma^{n-k}}$ is a strictly convex translation solution of \eqref{eq:CF} in $\R^{n-k+1}$.
\end{cor}
Huisken-Sinestrari obtained Theorem \ref{thm:CE} for the mean curvature flow in the case $m=1$, making spectacular use of it through their surgery program \cite{HuSi09}, which yielded a classification of 2-convex hypersurfaces.

Moreover, the $m=0$ case produces an analogue of Huisken's curvature estimate for convex solutions of the mean curvature flow \cite[Theorem 5.1]{Hu84}. 
This estimate implies that a convex solution of \eqref{eq:CF} becomes round at points of large curvature, which is crucial in proving that solutions contract to round points. This result was proved by different means for the class of flows considered here \cite{An94}.



\section{Preliminaries}\label{sec:prelims}

We will follow the notation used in \cite{hypersurfaces}. In particular, we recall that a smooth, symmetric function $g$ of the principal curvatures gives rise to a smooth function $G$ of the components of the Weingarten map. Equivalently, $G$ is an orthonormal frame invariant function of the components $\W_{ij}$ of the second fundamental form. To simplify notation, we denote $G(x,t)\equiv G\lb\W(x,t)\rb = g(\kappa(x,t))$ and use dots to denote derivatives of functions of curvature as follows:
\bann
\dot g^k(z)v_k=\left.\frac{d}{ds}\right|_{s=0}g(z+sv)\quad &\quad \dot G^{kl}(A)B_{kl}=\left.\frac{d}{ds}\right|_{s=0}G(A+sB)\\
\ddot g^{pq}(z)v_pv_q=\left.\frac{d^2}{ds^2}\right|_{s=0}g(z+sv)\quad &\quad \ddot G^{pq,rs}(A)B_{pq}B_{rs}=\left.\frac{d^2}{ds^2}\right|_{s=0}G(A+sB)\,.
\eann

The derivatives of $g$ and $G$ are related in the following way (cf. \cite{Ge90,An94,An07}):

\begin{lem}\label{lem:Fderivatives}
Let $g:\Gamma\to\R$ be a smooth, symmetric function. Define the function $G:\mathcal{S}_\Gamma:\to\R$ by $G(A):=g(\lambda(A))$, where $\lambda(A)$ denotes the eigenvalues of $A$ (up to order). Then for any diagonal $A$ with eigenvalues in $\Gamma$ we have 
\ba
\label{eq:DF} \dot G^{kl}(A) ={}&\dot g^k(\lambda(A))\delta^{kl}\,,
\ea 
and for any diagonal $A$ with distinct eigenvalues lying in $\Gamma$, and any symmetric $B\in GL(n)$, we have
\ba \label{E:D2F} \ddot G^{pq,rs}(A)B_{pq}B_{rs} ={}&\ddot
g^{pq}(\lambda(A))B_{pp}B_{qq}+2\sum_{p>q}\frac{\dot
g^p(\lambda(A))-\dot
g^q(\lambda(A))}{\lambda_p(A)-\lambda_q(A)}\big(B_{pq}\big)^2\,.
\ea
\end{lem}

In particular, in an orthonormal frame of eigenvectors of $\WW$ we have
\bann
\dot G^{kl}=&\dot g^k\delta^{kl}\\
\ddot G^{pq,rs}B_{pq}B_{rs} ={}&\ddot g^{pq}B_{pp}B_{qq}+2\sum_{p>q}\frac{\dot g^p-\dot g^q}{\kappa_p-\kappa_q}\big(B_{pq}\big)^2\,,
\eann
where we are denoting $\dot G\equiv \dot G\circ \W$, etc.

We note that $\ddot g\geq 0$ if and only if $(\dot g^p-\dot g^q)(z_p-z_q)\geq 0$ \cite[Lemma 2.2]{hypersurfaces}, so Lemma \ref{lem:Fderivatives} implies that $G$ is convex if and only if $g$ is convex.

\begin{lem}\label{lem:homogeneousbounds}
Let $X:M\times [0,T)\to\R^{n+1}$ be a solution of \eqref{eq:CF} such that $f$ satisfies Conditions (\ref{cond:homogeneity})--(\ref{cond:convexity}). Let $g:\Gamma\to\R$ be a smooth, degree zero homogeneous symmetric function. Then there exists $c>0$ such that
\bann
-c\leq g\lb \kappa_1(x,t),\dots,\kappa_n(x,t)\rb\leq c\,.
\eann
for all $(x,t)\in M\times [0,T)$. 

If $g>0$, then there exists $c>0$ such that
\bann
\frac{1}{c}\leq g\lb \kappa_1(x,t),\dots,\kappa_n(x,t)\rb\leq c\,.
\eann
\end{lem}
\begin{proof}
Let $\Gamma_0$ be a preserved cone for the solution $X$. Then $K:=\overline \Gamma_0\cap S^n$ is compact. Since $g$ is continuous, the required bounds hold on $K$. But these extend to $\overline \Gamma_0\setminus\{0\}$ by homogeneity. The claim follows since $\kappa(x,t)\in \overline \Gamma_0\setminus\{0\}$ for all $(x,t)\in M\times [0,T)$.
\end{proof}

By Condition (\ref{cond:homogeneity}), the derivative $\dot f$ of $f$ is homogeneous of degree zero. Since $\dot f^k>0$ for each $k$, we obtain uniform parabolicity of the flow:

\begin{cor}\label{cor:uniformparabolicity}
There exists a constant $c>0$ such that for any $v\in T^\ast M$ it holds that
\bann
\frac{1}{c}|v|^2 \leq \dot F^{ij}v_iv_j \leq c|v|^2\,,
\eann
where $| \cdot |$ is the (time-dependent) norm on $M$ corresponding to the (time-dependent) metric induced by the flow.
\end{cor}

We now recall the following evolution equation (see for example \cite{AnMcZh}):

\begin{lem}
Let $X:M\times[0,T)\to \R^{n+1}$ be a solution of \eqref{eq:CF} such that $f$ satisfies Conditions (\ref{cond:homogeneity})--(\ref{cond:convexity}). Let $G:M\times [0,T)\to \R$ be given by a smooth, symmetric, degree one homogeneous function $g$ of the principal curvatures. Then $G$ satisfies the following evolution equation:
\ba\label{eq:evolvG}
(\pd_t-\eL)G= (\dot G^{kl}\ddot F^{pq,rs}-\dot F^{kl}\ddot G^{pq,rs})\cd\W_{pq}\cd\W_{rs}+G|\WW|^2_F
\ea
where $\eL:=\dot F^{kl}\cd_k\cd_l$ is the linearisation of $F$, and $|\WW|^2_F:= \dot F^{kl}{\W_k}^r\W_{rl}$.
\end{lem}

In particular, the speed function $F$ satisfies
\bann
(\pd_t-\eL)F=F|\WW|^2_F\,.
\eann

As we shall see, in order to obtain Theorem \ref{thm:CE}, it is crucial to obtain a good upper bound on the term
\bann
Q(\cd\WW,\cd\WW):=(\dot G^{kl}\ddot F^{pq,rs}-\dot F^{kl}\ddot G^{pq,rs})\cd_k\W_{pq}\cd_l\W_{rs}
\eann
for the pinching functions $G_m$ which we construct in the following section. The following decomposition of $Q$ is crucial in obtaining this bound.

\begin{lem}\label{lem:decomposition}
For any totally symmetric $T\in \R^n\otimes\R^n\otimes\R^n$, we have
\ba
(\dot G^{kl}\ddot F^{pq,rs}-\dot F^{kl}\ddot G^{pq,rs})\big|_{B}T_{kpq}T_{lrs} ={}&(\dot g^{k}\ddot f^{pq}-\dot f^{k}\ddot g^{pq})\big|_{z}T_{kpp}T_{kqq}\nonumber\\
&+2\sum_{p>q}\frac{(\dot f^p\dot g^q-\dot g^p\dot f^q)\big|_{z}}{z_p-z_q}\Big((T_{pqq})^2+(T_{qpp})^2\Big)\nonumber\\
&+2\sum_{k>p>q}(\vec g_{kpq}\times \vec f_{kpq})\big|_{z}\cdot \vec z_{kpq}(T_{kpq})^2\label{eq:decomp}
\ea
at any diagonal matrix $B$ with distinct eigenvalues $z_i$, where `$\times$' and `$\,\cdot$' are the three dimensional cross and dot product respectively, and we have defined the vectors
\bann
\vec f_{kpq}:=(\dot f^k,\dot f^p,\dot f^q)\,,\quad &\vec 	g_{kpq}:=(\dot g^k,\dot g^p,\dot g^q)\,,
\eann
\vspace{-0.5cm}
\bann
\mbox{and}\quad \vec z_{kpq}:={}&\left(\frac{z_p-z_q}{(z_k-z_p)(z_k-z_q)}\,,\;\frac{z_k-z_q}{(z_k-z_p)(z_p-z_q)}\,,\;\frac{z_k-z_p}{(z_p-z_q)(z_k-z_q)}\right)\,.
\eann
\end{lem}

\begin{proof}
Since $B$ is diagonal, Lemma \ref{lem:Fderivatives} yields (supressing the dependence on $B$)
\bann
(\dot G^{kl}\ddot F^{pq,rs}-\dot F^{kl}\ddot G^{pq,rs})T_{kpq}T_{lrs} ={}&\sum_{k,p,q}(\dot g^{k}\ddot f^{pq}-\dot f^{k}\ddot g^{pq})T_{kpp}T_{kqq}\\
&+2\sum_{k}\sum_{p>q}\left(\dot g^k\frac{\dot f^p-\dot f^q}{z_p-z_q}-\dot f^k\frac{\dot g^p-\dot g^q}{z_p-z_q}\right)(T_{kpq})^2\,.\\
\eann
We now decompose the second term into the terms satisfying $k=p$, $k=q$, $k>p$, $p>k>q$, and $q>k$ respectively:
\begin{align*}
\sum_{k}\sum_{p>q}&\left(\dot g^k\frac{\dot f^p-\dot f^q}{z_p-z_q}-\dot f^k\frac{\dot g^p-\dot g^q}{z_p-z_q}\right)(T_{kpq})^2\\
={}&\sum_{p>q}\left(\dot g^p\frac{\dot f^p-\dot f^q}{z_p-z_q}-\dot f^p\frac{\dot g^p-\dot g^q}{z_p-z_q}\right)(T_{ppq})^2+\sum_{p>q}\left(\dot g^q\frac{\dot f^p-\dot f^q}{z_p-z_q}-\dot f^q\frac{\dot g^p-\dot g^q}{z_p-z_q}\right)(T_{qpq})^2\\
&+\left(\sum_{k>p>q}+\sum_{p>k>q}+\sum_{p>q>k}\right)\left(\dot g^k\frac{\dot f^p-\dot f^q}{z_p-z_q}-\dot f^k\frac{\dot g^p-\dot g^q}{z_p-z_q}\right)(T_{kpq})^2\\
={}&\sum_{p>q}\frac{\dot f^p\dot g^q-\dot g^p\dot f^q}{z_p-z_q}\Big((T_{pqq})^2+(T_{qpp})^2\Big)+\sum_{k>p>q}\left(\dot g^k\frac{\dot f^p-\dot f^q}{z_p-z_q}-\dot f^k\frac{\dot g^p-\dot g^q}{z_p-z_q}\right.\\
&+\left.\dot g^p\frac{\dot f^k-\dot f^q}{z_k-z_q}-\dot f^p\frac{\dot g^k-\dot g^q}{z_k-z_q}+\dot g^q\frac{\dot f^k-\dot f^p}{z_k-z_p}-\dot f^q\frac{\dot g^k-\dot g^p}{z_k-z_p}\right)(T_{kpq})^2\\
={}&\sum_{p>q}\frac{\dot f^p\dot g^q-\dot g^p\dot f^q}{z_p-z_q}\Big((T_{pqq})^2+(T_{qpp})^2\Big)+\sum_{k>p>q}\left((\dot g^p\dot f^q-\dot f^q\dot g^p)\left(\frac{1}{z_k-z_p}-\frac{1}{z_k-z_q}\right)\right.\\
&-\left.(\dot g^k\dot f^q-\dot f^k\dot g^q)\left(\frac{1}{z_p-z_q}+\frac{1}{z_k-z_p}\right)+(\dot g^k\dot f^p-\dot f^k\dot g^p)\left(\frac{1}{z_p-z_q}-\frac{1}{z_k-z_q}\right)\right)(T_{kpq})^2\\
={}&\sum_{p>q}\frac{\dot f^p\dot g^q-\dot g^p\dot f^q}{z_p-z_q}\Big((T_{pqq})^2+(T_{qpp})^2\Big)+\sum_{k>p>q}(\vec g_{kpq}\times \vec f_{kpq})\cdot \vec z_{kpq}(T_{kpq})^2\,.
\end{align*}
\end{proof}

We complete this section by proving that $(m+1)$-convexity is preserved by the flow \eqref{eq:CF}, so that this assumption need only be made on initial data:
\begin{prop}\label{prop:kconvexpreserved}
Let $X$ be a solution of \eqref{eq:CF} such that Conditions (\ref{cond:homogeneity})--(\ref{cond:convexity}) are satisfied. Suppose that there is some $m\in\{1,\dots,n-1\}$ and some $\beta>0$ such that 
$$
\kappa_{\sigma(1)}(x,0)+\dots+\kappa_{\sigma(m)}(x,0)\geq \beta F(x,0)
$$
for all $x\in M$ and all permutations $\sigma\in P_n$. Then this estimate persists at all later times.
\end{prop}
\begin{proof}
Denote by $SM$ the unit tangent bundle over $M\times[0,T)$ and consider the function $Z$ defined on $\oplus^m SM$ 
 by 
$$
Z(x,t,\xi_1,\ldots\xi_m)=\sum_{\alpha=1}^m h(\xi_\alpha,\xi_\alpha)-\beta F(x,t)\,.
$$
Since we have
$$
\inf_{\xi_1,\dots,\xi_m\in S_{(x,t)}M}Z(x,t,\xi_1,\dots,\xi_m)=\kappa_{\sigma(1)}(x,t)+\dots+\kappa_{\sigma(m)}(x,t)-\beta F(x,t)
$$
for some $\sigma\in P_n$, it suffices to show that $Z$ remains non-negative. First fix any $t_1\in [0,T)$ and consider the function $Z_\varepsilon(x,t,\xi_1,\ldots\xi_m):=Z(x,t,\xi_1,\ldots\xi_m)+\varepsilon \E^{(1+C)t}$, where $C:=\sup_{M\times [0,t_1]}|\WW|^2_F$. Note that $C$ is finite since $M$ is compact and $\dot F$ is bounded. Observe that $Z_\varepsilon$ is positive when $t=0$. We will show that $Z_\varepsilon$ remains positive on $M\times [0,t_1]$ for all $\varepsilon>0$. So suppose to the contrary that $Z_\varepsilon$ vanishes at some point $(x_0,t_0,\xi^0_1,\ldots\xi^0_m)$. We may assume that $t_0$ is the first such time. 
Now extend the vector $\xi^0:=(\xi^0_1,\ldots\xi^0_m)$ to a field $\xi:=(\xi_1,\dots,\xi_n)$ near $(x_0,t_0)$ by parallel translation in space and solving
\bann
\frac{\partial \xi_\alpha^i}{\partial t}=F\xi_\alpha^j{\W_j}^i\,.
\eann
Since the metric evolves according to
\bann
\pd_tg_{ij}=-2h_{ij}
\eann
the resulting fields have unit length. Now recall (see for example \cite{hypersurfaces}) the following evolution equation for the second fundamental form:
\bann
\pd_t\W_{ij}={}&\eL \W_{ij}+\ddot F^{pq,rs}\cd_i\W_{pq}\cd_j\W_{rs}+|\WW|^2_F\W_{ij}-2Fh^2_{ij}\,,
\eann
where $\eL:=\dot F^{kl}\cd_k\cd_l$ and $|\WW|^2_F:=\dot F^{kl}\W^2_{kl}$. It follows that
\bann
(\pd_t-\eL)\lb Z_\varepsilon(x,t,\xi) \rb={}&\varepsilon (1+C)\E^{(1+C)t}+\sum_{\alpha=1}^m\ddot F^{pq,rs}\cd_{\xi_\alpha}\W_{pq}\cd_{\xi_\alpha}\W_{rs}+|\WW(x,t)|^2_FZ(x,t,\xi)\\
\geq{}&\varepsilon(1+C)\E^{(1+C)t}+|\WW(x,t)|^2_FZ(x,t,\xi)\,.
\eann
Since the point $(x_0,t_0,\xi_{t=t_0})$ is a minimum of $Z_\varepsilon$, we obtain
\bann
0\geq (\pd_t-\eL)\big|_{(x_0,t_0)}\lb Z_\varepsilon (x,t,\xi)\rb \geq {}&\varepsilon (1+C)\E^{(1+C)t_0}-C\varepsilon\E^{(1+C)t_0}=\varepsilon\E^{(1+C)t_0}>0\,.
\eann
This is a contradiction, implying that $Z_\varepsilon$ cannot vanish at any time in the interval $[0,t_1]$. Since $\varepsilon>0$ was arbitrary, we find $Z\geq 0$ at all times in the interval $[0,t_1]$. Since $t_1\in[0,T)$ was arbitrary, we obtain $Z\geq 0$.

\end{proof}


\section{Constructing the pinching function.}

In this section we construct the pinching functions $G_m$ satisfying the conditions in Theorem \ref{thm:CE}. 
Let us first introduce the `pinching cones' 
$$
\Gamma_m:=\{z\in \Gamma: z_{\sigma(1)}+\dots+z_{\sigma(m+1)}> c_m^{-1}f(z)\;\mbox{for all}\; \sigma\in H_m\}\,,
$$
where $H_m$ is the quotient of $P_n$, the group of permutations of the set $\{1,\dots,n\}$, by the equivalence relation
$$
\sigma \sim \omega \quad\mbox{if}\quad \sigma\lb\lcb 1,\dots,m+1\rcb\rb=\omega\lb\lcb 1,\dots,m+1\rcb\rb\,.
$$
Using the methods of \cite{Hu84}, and their adaptations to two-convex flows in \cite{HuSi09} and fully non-linear flows in \cite{hypersurfaces}, we will see that, in order to prove Theorem \ref{thm:CE}, it suffices to construct a smooth function $g_m:\Gamma\to\R$ satisfying the following properties:
\begin{props*}
\mbox{}
\ben[(i)]
\item\label{prop:pinching} $g_m(z)\geq 0$ for all $z\in \Gamma$ with equality if and only if $z\in \overline\Gamma_{m}\cap\Gamma$;
\item\label{prop:homogeneous} $g_m$ is smooth an homogeneous of degree one;
\item\label{Qest} for every $\varepsilon>0$ there exists $c_\varepsilon>0$ such that for all diagonal matrices $B$ and totally symmetric 3-tensors $T$, it holds that
\bann
(\dot G_m^{kl}\ddot F^{pq,rs}-\dot F^{kl}\ddot G_m^{pq,rs})\big|_{B}T_{kpq}T_{lrs}\leq-c_\varepsilon\frac{|T|^2}{F}
\eann
for all symmetric matrices $B$ satisfying $\lambda(B)\in\Gamma_0$, and $G_m(B)\geq\varepsilon F(B)$, where $G_m$ is the matrix function corresponding to $g_m$ as described in Section \ref{sec:prelims}, and $\Gamma_0$ is a preserved cone for the flow; and
\item\label{Zest} for every $\delta>0$, $\varepsilon>0$, and $C>0$, there exist $\gamma_\varepsilon>0$ and $\gamma_\delta>0$ such that
\bann
(G_m\dot F^{kl}-F\dot G_m^{kl})\big|_{B}B^2_{kl}\leq-\gamma_\varepsilon F^2(G_m-\delta F)\big|_{B}+\gamma_\delta F^2\big|_{B}
\eann
for all symmetric, $(m+1)$-positive matrices $B$ satisfying $\lambda(B)\in\Gamma_0$, $G_m(B)\geq\varepsilon F(B)$, and $\lambda_{\min}(B)\geq -\delta F(B)-C$.
\een
\end{props*}

Our construction of the pinching function $g_m$ will be independent of the choice of $m$. So let us fix $m\in\{0,1,\dots,n-2\}$ and assume that the flow is $(m+1)$-convex. We first consider the preliminary function $g:\Gamma \to \R$ defined by
\ba\label{eq:g1}
g(z):=f(z)\sum_{\sigma\in H_m}\varphi\left(\frac{\sum_{i=1}^{m+1}z_{\sigma(i)}-\frac{1}{c_m}f(z)}{f(z)}\right)\,,
\ea
where $\varphi:\R\to\R$ is a smooth\footnote{In fact, $\varphi$ need only be twice continuously dfferentiable.} function which is strictly convex and positive, except on $\R_+\cup\{0\}$, where it vanishes identically. Such a function is readily constructed; for example, we could take
\bann
\varphi(r)={}&\begin{dcases}
                   \mbox{}\; r^4e^{-\frac{1}{r^2}} & \mbox{if} \quad r< 0\\
                   \mbox{}\; 0 & \mbox{if} \quad r\geq 0\,.
                \end{dcases}
\eann 
We note that such a function necessarily satisfies $\varphi(r)-r\varphi'(r)\leq 0$ and $\varphi'(r)\leq 0$ with equality if and only if $r\geq 0$.

Now define the scalar $G:M\times[0,T)\to\R$ by $G(x,t):=g(\kappa_1(x,t),\dots,\kappa_n(x,t))$. Then $G$ is a smooth, degree one homogeneous function of the components of the Weingarten map which is invariant under a change of basis. Moreover, $G$ is non-negative and vanishes at, and only at, points for which the sum of the smallest $(m+1)$-principal curvatures is not less than $c_m^{-1}F$. Thus Properties (\ref{prop:pinching}) and (\ref{prop:homogeneous}) are satisfied by $g$.

We now show that property (\ref{Qest}) is satisfied weakly by $g$:
\begin{lem}\label{lem:Qest}
Let $G$ be the matrix function corresponding to the function $g$ defined by \eqref{eq:g1}. Then for any diagonal matrix $B$ and totally symmetric 3-tensor $T$, it holds that
\bann
(\dot G^{kl}\ddot F^{pq,rs}-\dot F^{kl}\ddot G^{pq,rs})\big|_{B}T_{kpq}T_{lrs}\leq 0
\eann
\end{lem}
\begin{proof}
We will show that each of the terms in the decomposition \eqref{eq:decomp} in Lemma \ref{lem:decomposition} is non-positive. Note that it suffices to compute at matrices having distinct eigenvalues, since the result at an arbitrary symmetric matrix $B$ may be obtained by taking a limit $B^{(k)}\to B$ such that each matrix $B^{(k)}$ has distinct eigenvalues. Thus we may assume that the eigenvalues satisfy $z_1<\dots <z_n$. We first compute,
\bann
\dot g^k ={}&\dot f^k\sum_{\sigma\in H_m}\varphi\left(r_\sigma\right)+\sum_{\sigma\in H_m}\varphi'\left(r_\sigma\right)\sum_{i=1}^{m+1}\left({\delta_{\sigma(i)}}^k -\frac{z_{\sigma(i)}}{f}\dot f^k\right)\\
={}&\dot f^k\sum_{\sigma\in H_m}\lb \varphi\left(r_\sigma\right)-\varphi'\left(r_\sigma\right)\frac{\sum_{i=1}^{m+1}z_{\sigma(i)}}{f}\rb +\sum_{\sigma\in H_m}\sum_{i=1}^{m+1}\varphi'\left(r_\sigma\right){\delta_{\sigma(i)}}^k\,,\\
\ddot g^{pq} ={}&\left(\sum_{\sigma\in H_m}\varphi\left(r_\sigma\right)-\sum_{\sigma\in H_m}\varphi'\left(r_\sigma\right)\frac{\sum_{i=1}^{m+1}z_{\sigma(i)}}{f}\right)\ddot f^{pq}\\
&+\sum_{\sigma\in H_m}\frac{\varphi''(r_\sigma)}{f}\sum_{i=1}^{m+1}\left({\delta_{\sigma(i)}}^p-\frac{z_{\sigma(i)}}{f}\dot f^p\right)\sum_{i=1}^{m+1}\left({\delta_{\sigma(i)}}^q-\frac{z_{\sigma(i)}}{f}\dot f^q\right)\,,
\eann
where we are denoting $r_\sigma(z):=\frac{\sum_{i=1}^{m+1}z_{\sigma(i)}-c_m^{-1}f(z)}{f(z)}$. It follows that
\bann
\dot g^{k}\ddot f^{pq}-\dot f^{k}\ddot g^{pq} ={}&\sum_{\sigma\in H_m}\sum_{i=1}^{m+1}\varphi'(r_\sigma){\delta_{\sigma(i)}}^k\ddot f^{pq}\\
&-\dot f^k\sum_{\sigma\in H_m}\frac{\varphi''(r_\sigma)}{f}\sum_{i=1}^{m+1}\left({\delta_{\sigma(i)}}^p-\frac{z_{\sigma(i)}}{f}\dot f^p\right)\sum_{i=1}^{m+1}\left({\delta_{\sigma(i)}}^q-\frac{z_{\sigma(i)}}{f}\dot f^q\right)\,.
\eann
If we fix the index $k$ and set $\xi_p=T_{kpp}$, then, by convexity of $\varphi$ and positivity of $\dot f^k$, we have
\bann
-\dot f^k\sum_{\sigma\in H_m}\frac{\varphi''(r_\sigma)}{f}\sum_{i=1}^{m+1}\left({\delta_{\sigma(i)}}^p-\frac{z_{\sigma(i)}}{f}\dot f^p\right)&\sum_{i=1}^{m+1}\left({\delta_{\sigma(i)}}^q-\frac{z_{\sigma(i)}}{f}\dot f^q\right)\xi_p\xi_q\\
 ={}&-\dot f^k\sum_{\sigma\in H_m}\frac{\varphi''(r_\sigma)}{f}\left(\sum_{i=1}^{m+1}\left({\delta_{\sigma(i)}}^p-\frac{z_{\sigma(i)}}{f}\dot f^p\right)\xi_p\right)^2\\
 \leq{}&0\,.
\eann
On the other hand, since $\varphi$ is monotone non-increasing, and $f$ is convex, we have
\bann
\varphi'(r_\sigma)\sum_{i=1}^{m+1}{\delta_{\sigma(i)}}^k\ddot f^{pq}\xi_p\xi_q \leq {}&0
\eann
for each $\sigma$. Since both inequalities hold for all $k$, we deduce that
\bann
\sum_{k,p,q}\big(\dot g^{k}\ddot f^{pq}-\dot f^{k}\ddot g^{pq}\big)T_{kpp}T_{kqq} \leq{}&0\,.
\eann
We next consider
\bann
\dot f^p\dot g^q-\dot g^p\dot f^q ={}&\sum_{\sigma\in H_m}\sum_{i=1}^{m+1}\varphi'(r_\sigma)\lb{\delta_{\sigma(i)}}^q\dot f^p-{\delta_{\sigma(i)}}^p\dot f^q\rb\nonumber\\
={}&\lb\sum_{\sigma\in O_q}\varphi'(r_\sigma)\dot f^p-\sum_{\sigma\in O_p}\varphi'(r_\sigma)\dot f^q\rb
\eann
Thus, if $z_p>z_q$, we obtain
\bann
\dot f^p\dot g^q-\dot g^p\dot f^q\leq {}&\dot f^p\lb\sum_{\sigma\in O_q}\varphi'(r_\sigma)-\sum_{\sigma\in O_p}\varphi'(r_\sigma)\rb\,.\\
\eann
where we have introduced the sets $O_a:=\{\sigma\in H_m:a\in \sigma(\{1,\dots,m+1\})\}$. We now show that the term in brackets is non-positive whenever $z_p>z_q$:
\begin{lem}\label{lem:pairing}
If $z_p>z_q$, then
\bann
\sum_{\sigma\in O_{p}}\varphi'(r_\sigma)-\sum_{\sigma\in O_{q}}\varphi'(r_\sigma)\geq 0\,.
\eann
\end{lem}
\begin{proof}[Proof of Lemma \ref{lem:pairing}]
First note that
\bann
\sum_{\sigma\in O_{p}}\varphi'(r_\sigma)-\sum_{\sigma\in O_{q}}\varphi'(r_\sigma)=\sum_{\sigma\in O_{p,q}}\varphi'(r_\sigma)-\sum_{\sigma\in O_{q,p}}\varphi'(r_\sigma)\,,
\eann
where $O_{a,b}:=O_a\setminus O_b$. Next observe that, if $\sigma\in O_{p,q}$, then 
\ba\label{eq:pairing}
z_{\sigma(1)}+\dots+z_{\sigma(m+1)}=z_{p}+z_{\hat\sigma(i_1)}\dots+z_{\hat \sigma(i_m)}
\ea
for some $\hat\sigma\in H_{m-2}(p,q):=P_{n-2}(p,q)/\sim$, where $P_{n-2}(p,q)$ is the set of permutations of $\{1,\dots,n\}\setminus\{p,q\}$, $\{i_1,\dots,i_{m}\}$ are a choice of $m$ elements of $\{1,\dots,n\}\setminus\{p,q\}$, and $\sim$ is defined by
$$
\hat\sigma\sim\hat\omega\quad\mbox{if}\quad \hat\sigma(\{i_1,\dots,i_m\})=\hat\omega(\{i_1,\dots,i_m\})\,.
$$
Observe also that the converse holds (that is, \eqref{eq:pairing} defines a bijection), so that
\bann
\sum_{\sigma\in O_{q,p}}\varphi'(r_\sigma)-\sum_{\sigma\in O_{p,q}}\varphi'(r_\sigma)={}&\sum_{\hat\sigma\in H_{m-2}(p,q)}\bigg[\varphi'\lb\frac{z_p+\sum_{k=1}^{m}z_{\hat\sigma(i_k)}-c_m^{-1}f}{f}\rb\\{}&\qquad\qquad\qquad-\varphi'\lb\frac{z_q+\sum_{k=1}^{m}z_{\hat\sigma(i_k)}-c_m^{-1}f}{f}\rb\bigg]\,.
\eann
Since $z_p>z_q$ the claim follows from convexity of $\varphi$.
\end{proof}
Thus,
\bann
\sum_{p>q}\frac{\dot f^p\dot g^q-\dot g^p\dot f^q}{z_p-z_q}\Big((T_{pqq})^2+(T_{qpp})^2\Big) \leq{}&0\,.
\eann
We now compute
\bann
\vec g_{kpq} ={}&\left(\frac{g}{f}-\sum_{\sigma\in H_m}\varphi'(r_\sigma)\sum_{i=1}^{m+1}\frac{z_{\sigma(i)}}{f}\right)\vec f_{kpq}+\sum_{\sigma\in H_m}\varphi'(r_\sigma)\sum_{i=1}^{m+1}\big({\delta_{\sigma(i)}}^k,{\delta_{\sigma(i)}}^p,{\delta_{\sigma(i)}}^q\big)\,,
\eann
so that
\bann
\Big(\vec g_{kpq}\times\vec f_{kpq}\Big)\cdot\vec z_{kpq} ={}&\sum_{\sigma\in H_m}\sum_{i=1}^{m+1}\varphi'(r_\sigma)\left[\big({\delta_{\sigma(i)}}^k,{\delta_{\sigma(i)}}^p,{\delta_{\sigma(i)}}^q\big)\times\vec f_{kpq}\right]\cdot\vec z_{kpq}\\
 ={}&\sum_{\sigma\in H_m}\sum_{i=1}^{m+1}\varphi'(r_\sigma)\left[\frac{({\delta_{\sigma(i)}}^p\dot f^q-{\delta_{\sigma(i)}}^q\dot f^p)(z_p-z_q)}{(z_k-z_p)(z_k-z_q)}\right.\\
&\hspace{3cm} +\frac{({\delta_{\sigma(i)}}^q\dot f^k-{\delta_{\sigma(i)}}^k\dot f^q)(z_k-z_q)}{(z_k-z_p)(z_p-z_q)}\\
&\left.\hspace{3cm}+\frac{({\delta_{\sigma(i)}}^k\dot f^p-{\delta_{\sigma(i)}}^p\dot f^k)(z_k-z_p)}{(z_k-z_q)(z_p-z_q)}\right].
\eann
Removing the positive factor $\alpha_{kpq}:=[(z_k-z_p)(z_z-z_q)(z_p-z_q)]^{-1}$ and setting $P_{a}:=\sum_{\sigma\in O_a}\varphi'(r_\sigma)$, we obtain
\bann
\Big(\vec g_{kpq}\times\vec f_{kpq}\Big)\cdot\vec z_{kpq} ={}& \alpha_{kpq}\Big[(P_p\dot f^q-P_q\dot f^p)(z_p-z_q)^2+(P_q\dot f^k-P_k\dot f^q)(z_p-z_q)^2\\
{}&\qquad+(P_k\dot f^p-P_p\dot f^k)(z_p-z_q)^2\Big]\,.
\eann
Applying Lemma \ref{lem:pairing} yields
\bann
\Big(\vec g_{kpq}\times\vec f_{kpq}\Big)\cdot\vec z_{kpq} \leq{}& \alpha_{kpq}\lb P_q\dot f^k-P_k\dot f^q\rb\lsb (z_k-z_q)^2-(z_k-z_p)^2-(z_p-z_q)^2\rsb\,.
\eann
Since the term in square brackets is non-negative, applying Lemma \ref{lem:pairing} once more yields
$$
\Big(\vec g_{kpq}\times\vec f_{kpq}\Big)\cdot\vec z_{kpq} \leq 0\,.
$$

This completes the proof of the lemma.
\end{proof}

In particular, Lemma \ref{lem:Qest} yields an upper bound for $G/F$ along the flow:
\begin{cor}
There exists $C_1<\infty$ such that $G/F\leq C_1$ along the flow.
\end{cor}
\begin{proof}
In view of Lemma \ref{lem:Qest} and the evolution equation \eqref{eq:evolvG} this is a simple application of the maximum principle.
\end{proof}

In order to obtain the uniform estimate required by property (\ref{Qest}), we modify $G$ in order to obtain a function with a strictly positive term in $Q$. A well-known trick (cf. \cite[Theorem 2.14]{HuSi99b}, \cite[Lemma 3.3]{hypersurfaces}) then allows us to extract the required uniform estimate. First, we relabel the preliminary pinching funtion $g\to g_1$ ($G\to G_1$), and consider the new pinching function $g$ defined by:
\ba\label{eq:g}
g:={}&K(g_1,g_2)\;:=\;\frac{g_1^2}{g_2}\,,
\ea
where $g_2(z)=M\sum_{i=1}^nz_i-|z|$ for some large constant $M>>1$, for which $g_2$ is positive along the flow. That there is such a constant follows from applying the maximum principle to the evolution equation \eqref{eq:evolvG} for the function $G_2(x,t):=g_2(\kappa(x,t))$ as in \cite[Lemma 3.1]{hypersurfaces}. Note that $\dot K^1>0$, $\dot K^2<0$ and $\ddot K>0$ wherever $g_1>0$.

Observe that Properties (\ref{prop:pinching}) and (\ref{prop:homogeneous}) are not harmed in the transition from $g_1$ to $g$. We now show that the estimates listed in Properties (\ref{Qest}) and (\ref{Zest}) are satisfied by the curvature function defined in \eqref{eq:g}.

\begin{prop}\label{prop:Qest}
Let $g$ be the pinching function defined by \eqref{eq:g} and $G$ its corresponding matrix function. Then, for every $\varepsilon>0$ there exists $c_\varepsilon>0$ such that for all diagonal matrices $B$ and totally symmetric 3-tensors $T$, it holds that
\bann
(\dot G^{kl}\ddot F^{pq,rs}-\dot F^{kl}\ddot G^{pq,rs})\big|_{B}T_{kpq}T_{lrs}\leq-c_\varepsilon\frac{|T|^2}{F}
\eann
whenever $G(B)\geq\varepsilon F(B)$.
\end{prop}
\begin{proof}
First note that (supressing dependence on $B$)
\bann
(\dot G^{kl}\ddot F^{pq,rs}-\dot F^{kl}\ddot G^{pq,rs})T_{kpq}T_{lrs}={}&\dot K^\alpha (\dot G_\alpha^{kl}\ddot F^{pq,rs}-\dot F^{kl}\ddot G_\alpha^{pq,rs})T_{kpq}T_{lrs}\\
{}&-\dot F^{kl}\ddot K^{\alpha\beta}\dot G_\alpha^{pq}\dot G_\beta^{rs}T_{kpq}T_{lrs}\\
\leq{}&\dot K^2(\dot G_2^{kl}\ddot F^{pq,rs}-\dot F^{kl}\ddot G_2^{pq,rs})T_{kpq}T_{lrs}\\
\leq{}&-\dot K^2\dot F^{kl}\ddot G_2^{pq,rs}T_{kpq}T_{lrs}\,,
\eann
where we used Lemma \ref{lem:Qest}, convexity of $K$, and the inequalities $\dot K^1\geq 0$ and $\dot F\geq 0$ in the first inequality, and the inequalities $\dot G_2\geq 0$ and $\dot K^2\leq 0$, and convexity of $F$ in the second. Since $\dot K^2<0$ whenever $G_1>0$ and $G_2$ is strictly concave in non-radial directions, the claim follows from a well-known trick, exactly as in \cite[Lemma 3.3]{hypersurfaces}.
\end{proof}

The uniform estimate of Proposition \ref{prop:Qest} yields a good bound for the term $Q(\cd\WW,\cd\WW)$ in the evolution equation for the pinching functions $G$. This is a crucial component in obtaining the $L^p$-estimates of the follwing section. 
These are the starting point for the Stampacchia-De Giorgi iteration argument. The second crucial estimate is the Poincar\'e-type inequality, Lemma \ref{lem:Lpestimate} (see also sections 4 and 5 of \cite{HuSi09}; in particular, Lemma 5.5), which we can obtain with the help of property (\ref{Zest}). This estimate (corresponding to Lemma 5.2 of \cite{HuSi09}) provides an estimate on the zero order term that occurs in contracting the Simons-type identity for $\dot F^{pq}\cd_p\cd_q\W_{ij}$ with $\dot G^{ij}$ (cf. \cite[Proposition 4.4]{hypersurfaces}).

\begin{prop}\label{prop:Zest}
Let $g$ be the pinching function defined by \eqref{eq:g} and $G$ its corresponding matrix function. Then, for every $\delta>0$, $\varepsilon>0$, and $C_{\delta}>0$ there exist $\gamma_1>0$ and $\gamma_2>0$ such that
\bann
(F\dot G^{kl}-G\dot F^{kl})\big|_{B}B^2_{kl}\geq \gamma_\varepsilon F^2(G-\delta F)\big|_{B}-\gamma_\delta F^2\big|_{B}
\eann
for all symmetric, $(m+1)$-positive matrices $B$ satisfying $\lambda(B)\in\Gamma_0$, $G_m(B)\geq \varepsilon F(B)$, and $\lambda_{\min}(B)\geq -\delta F(B)-C_{\delta}$.
\end{prop}

\begin{proof}
So let $B$ be a symmetric, $(m+1)$-positive matrix with eigenvalues $z_1\leq \dots\leq z_n$. Define $Z(B):=F\dot G(B^2)-G\dot F(B^2)$. Then
\bann
Z(B)={}&f\dot g^pz^2_p-g\dot f^pz_p^2\\
={}&\sum_{p>q}\big(\dot g^p\dot f^q-\dot g^q\dot f^p\big)z_pz_q(z_p-z_q)\\
={}&\sum_{p>q}\big(P_p\dot f^q-P_q\dot f^p\big)z_pz_q(z_p-z_q)\\
={}&\lb\sum_{p>q>l}+\sum_{p>l\geq q}+\sum_{l\geq p> q}\rb\big(P_p\dot f^q-P_q\dot f^p\big)z_pz_q(z_p-z_q)\,,
\eann
where we recall the notation $P_a:=\sum_{\sigma\in O_a}\varphi'(r_\sigma)$ and we have defined $l\leq m$ as the number of non-positive eigenvalues $z_i$. Recalling that $P_p\dot f^q-P_q\dot f^p\geq0$ whenever $z_p\geq z_q$, we discard the final sum and part of the first to obtain
\bann
Z(B)\geq{}&\sum_{p=m+2}^n\sum_{q=l+1}^{m+1}\big(P_p\dot f^q-P_q\dot f^p\big)z_pz_q(z_p-z_q)+\sum_{p=l+1}^n\sum_{q=1}^l\big(P_p\dot f^q-P_q\dot f^p\big)z_pz_q(z_p-z_q)\,.
\eann
Observe that when $a\leq m+1$, we have
$$
P_a\leq \varphi'\lb\frac{z_1+\dots+z_{m+1}-c_m^{-1}f}{f}\rb\,,
$$
which is strictly negative: for it can only vanish if $z_1+\dots+z_{m+1}-c_m^{-1}f\geq 0$, in which case $G(B)=0$, which contradicts $G(B)\geq\varepsilon F(B)>0$. It follows that, for $q\leq m+1$, the term $P_p\dot f^q-P_q\dot f^p\geq \dot f^p(P_p-P_q)$ can only vanish if $P_p=P_q$, which will only occur if $z_p=z_q$ since $\varphi$ is strictly convex where it is positive (cf. Lemma \ref{lem:pairing}). Since $P_p\dot f^q-P_q\dot f^p$ is homogeneous of degree zero with respect to $z$, we obtain the uniform bound
\bann
\sum_{p=m+2}^n\sum_{q=l+1}^{m+1}\big(P_p\dot f^q-P_q\dot f^p\big)z_pz_q(z_p-z_q)\geq c\sum_{p=m+2}^n\sum_{q=l+1}^{m+1}z_pz_q(z_p-z_q)
\eann
for some $c>0$. On the other hand, again by homogeneity, the term $P_p\dot f^q-P_q\dot f^p$ is also bounded above (for all $p$, $q$), in which case we obtain
\bann
\sum_{p=l+1}^n\sum_{q=1}^l\big(P_p\dot f^q-P_q\dot f^p\big)z_pz_q(z_p-z_q)\geq C\sum_{p=l+1}^n\sum_{q=1}^lz_pz_q(z_p-z_q)
\eann
for some $C<\infty$. Agreeing to denote positive constants simply by $c$, we deduce
\ba\label{eq:Z}
Z(B)\geq{}&c\lb\sum_{p=l+1}^n\sum_{q=1}^lz_pz_q(z_p-z_q)+\sum_{p=m+2}^n\sum_{q=l+1}^{m+1}z_pz_q(z_p-z_q)\rb
\ea
We control the first sum using the `convexity estimate' $z_1\geq -\delta F-C_\delta$ as follows:
\ba
\sum_{p=l+1}^{n}\sum_{q=1}^lz_pz_q(z_p-z_q)\geq {}&(n-l)z_{n}\sum_{q=1}^lz_q(z_{n}-z_q)\\
\geq {}&2(n-l)c^2F^2\sum_{q=1}^lz_q\nonumber\\
\geq {}&-2(n-l)c^2F^2(\delta F+C_\delta)\nonumber\\
\equiv {}&-cF^2(\delta F+C_{\delta})\,,\label{eq:Zest1}
\ea
where we estimated $-c \leq z_i/F\leq c$ for each $i$.

Recall $m\leq n-2$. Then we may decompose the good second term in the brackets on the right hand side of \eqref{eq:Z} as
\bann
\sum_{p=m+2}^n\sum_{q=l+1}^{m+1}z_pz_q(z_p-z_q)={}&\lb\sum_{p=m+2}^n\sum_{q=l+1}^{m+1}z_pz_q(z_p-z_q)-F^2\sum_{k=1}^lz_k\rb+F^2\sum_{k=1}^lz_k
\eann
where $l$ is again the number of non-positive eigenvalues. Consider first the term in the brackets, $S_1:=\sum_{p=m+2}^n\sum_{q=l+1}^{m+1}z_pz_q(z_p-z_q)-F^2\sum_{k=1}^lz_k$. Since each of the terms is non-negative, $S_1$ can only vanish if $z_k=0$ for all $k\leq l$ and $z_p(z_p-z_q)=0$ for all $p>q>l$. That is, if there are no negative eigenvalues, and the positive ones (of which there are at least $n-m$) are all equal. But this implies $(z_1+\dots+z_{k+1})-c_m^{-1}f\geq 0$, which in turn implies $g=0<\varepsilon f$, a contradiction. We thus obtain a positive lower bound for the degree zero homogeneous quantity $S_1/(F^2G)$:
\bann
S_1 \geq {}&c F^2G
\eann
for some $c>0$. The remaining term is again easily estimated using the convexity estimate:
\bann
S_2:=F^2\sum_{k=1}^lz_k\geq {}&-cF^2(\delta F+C_{\delta})\,.
\eann
The claim follows.
\end{proof}

We note that the above estimate is only useful in the presence of the convexity estimate, Theorem \ref{thm:AC}, since in that case, for any $\delta>0$, there is a constant $C_\delta>0$ for which the set $\Gamma_{\delta,C_\delta}:=\{z\in \Gamma_0\,:\, z_i>-\delta f(z)-C_\delta\; \mbox{ for all }\; i\}$ is preserved by the flow.

\section{Proof of Theorem \ref{thm:CE}}

In order to prove Theorem \ref{thm:CE} it suffices to obtain for any $\varepsilon>0$ an upper bound on the function
$$
\Ges:=\lb \frac{G}{F}-\varepsilon\rb F^{\sigma}
$$
for some $\sigma>0$. We will use the estimates of Propositions \ref{prop:Zest} and \ref{prop:Qest} to obtain bounds on the space-time $L^p$-norms of the positive part of $\Ges$, so long as $p$ is sufficiently large and $\sigma$ sufficiently small, just as in \cite{HuSi99a,HuSi99b,HuSi09} (see also \cite{hypersurfaces} where these techniques are applied in the fully non-linear setting). The Stampacchia-De Giorgi iteration procedure introduced in \cite{Hu84} (see also \cite{HuSi99a,hypersurfaces}) then allows us to extract a supremum bound on $\Ges$.

We recall the following evolution equation from \cite{hypersurfaces}:
\begin{lem}
The function $\Ges$ satisfies the following evolution equation:
\ba
(\pd_t-\eL)\Ges ={}& F^{\sigma-1}(\dot G^{kl}\ddot F^{pq,rs}-\dot F^{kl}\ddot G^{pq,rs})\cd_k\W_{pq}\cd_l\W_{rs}+\frac{2(1-\sigma)}{F}\langle \cd \Ges,\cd F\rangle_F\nonumber\\
&-\frac{\sigma(1-\sigma)}{F^2}|\cd F|^2_F+\sigma
\Ges|\WW|^2_F\,,\label{eq:evolvGes} 
\ea
where $\langle u,v\rangle_F:=\dot F^{kl}u_ku_l$.
\end{lem}

Now set $E:=\max\{\Ges,0\}$. We need to obtain space-time $L^p$-estimates for $E$. Let us first observe that integration by parts and application of Young's inequality, in conjunction with Lemma \ref{lem:homogeneousbounds} and Proposition \ref{prop:Qest}, yields the estimate (cf. \cite{hypersurfaces})
\ba
\frac{d}{dt}\int E^pd\mu \leq &-\left(A_1p(p-1)-A_2p^{\frac{3}{2}}\right)\int E^{p-2}|\cd\Ges|^2\, \,d\mu\nonumber\\
& -\left(B_1p-B_2p^{\frac{1}{2}}\right)\int E^{p}\frac{|\cd \WW|^2}{F^2}\,d\mu+C_1\sigma p\int E^p|\WW|^2d\mu\label{eq:Lpestimate}
\ea
for some positive constants $A_1$, $A_2$, $B_1$, $B_2$, $C_1$ which are independent of $\sigma$ and $p$.

To estimate the final term, we use Proposition \ref{prop:Zest} in a similar manner to \cite[Section 5]{HuSi09}. We first observe:
\begin{lem}\label{lem:Lpestimate}
There are positive constants $A_3, A_4, A_5, B_3, B_4, C_2$ which are independent of $p$ and $\sigma$ such that: 
\bann
\int E^p\frac{Z(\WW)}{F}\,d\mu \leq{}&\big(A_3p^{\frac{3}{2}}+A_4p^{\frac{1}{2}}+A_5\big)\int E^{p-2}|\cd\Ges |^2\,d\mu+\big(B_3p^{\frac{1}{2}}+B_4\big)\int E^{p}\frac{|\cd\WW|^2}{F^{2}}\,d\mu\,.
\eann
\end{lem}
\begin{proof}
As in \cite[Section 4]{hypersurfaces}, contraction of the commutation formula for $\cd^2\WW$ with $\dot F$ and $\dot G$ yields the identity
\bann
\eL\Ges ={}& -F^{\sigma-1}Q(\cd\WW,\cd\WW)+F^{\sigma-1}Z(\WW)+F^{\sigma-2}(F\dot G^{kl}-G\dot F^{kl})\cd_k\cd_lF\\
{}&+\frac{\sigma}{F}\Ges \mathcal{L}F-2\frac{(1-\sigma)}{F}\langle\cd F,\cd\Ges\rangle_F+\frac{\sigma(1-\sigma)}{F^2}\Ges |\cd F|^2_F\,.
\eann
The claim is now proved using integration by parts 
and Young's inequality, with the help of Lemma \ref{lem:homogeneousbounds} and Propostion \ref{prop:Qest} (cf. \cite[Lemma 4.2]{hypersurfaces}).
\end{proof}

\begin{cor}
For all $\varepsilon>0$ there exist constants $\ell>0$ and $L>0$ such that for all $p>L$ and $0<\sigma <\ell p^{-\frac{1}{2}}$ there is a constant $K=K_{\varepsilon,\sigma,p}$ for which the following estimate holds:
\bann
\int (\Ges)_+^pd\mu\leq \int \lb \Ges(\cdot,0)\rb_+^pd\mu_0 + tK\mu_0(M)\,, 
\eann
where $\mu_0$ is the measure induced on $M$ by the initial immersion.
\end{cor}
\begin{proof}
Recall Proposition \ref{prop:Zest}. Setting $\delta=\varepsilon/2$ we obtain
\bann
\frac{Z(\WW)}{F}\geq \frac{\varepsilon}{2}\gamma_1 F^2-\gamma_{2} F
\eann
whenever $G-\varepsilon F>0$. By Young's inequality, for all $\sigma p>0$ there is a constant $K_{\sigma,p}$ such that
\bann
F\leq \sigma p F^2+ K_{\sigma,p} F^{-\sigma p}\,,
\eann
so that
\bann
\lb\frac{\varepsilon}{2}\gamma_1-\sigma p \gamma_2 \rb F^2\leq K_{\sigma,p} F^{-\sigma p}+\frac{Z(\WW)}{F}\,.
\eann
If we are careful to ensure $\sigma p \gamma_2\leq \varepsilon\gamma_1/4$, we obtain
\bann
\frac{\varepsilon\gamma_1}{4}F^2\leq K_{\sigma,p} F^{-\sigma p}+\frac{Z(\WW)}{F}\,.
\eann
Since $\Ges$ is bounded by $F^{\sigma}$, and $|\WW|^2$ is bounded by $F^2$, we obtain
\bann
E^p|\WW|^2\leq K_{\varepsilon,\sigma,p}+ c_\varepsilon E^p\frac{Z(\WW)}{F}\,,
\eann
for some constants $K_{\varepsilon,\sigma,p}>0$ depending on $\varepsilon$, $\sigma$ and $p$, and $c_\varepsilon>0$ depending on $\varepsilon$ (but independent of $\sigma$ and $p$). 

Combining Lemma \ref{lem:Lpestimate} and inequality \eqref{eq:Lpestimate} now  yields
\bann
\frac{d}{dt}\int E^pd\mu \leq {}& K_{\varepsilon,\sigma,p}\mu_0(M) -\left(\alpha_0p^2-\alpha_1\sigma p^{\frac{5}{2}}-\alpha_2 p^{\frac{3}{2}}-\alpha_3 p\right)\int E^{p-2}|\Ges|^2\, \,d\mu\nonumber\\
& -\left(\beta_0 p-\beta_1\sigma p^{\frac{3}{2}}-\beta_2\sigma p-\beta_3p^{\frac{1}{2}} \right)\int E^{p}\frac{|\cd\WW|^2}{F^2}\,d\mu\,. 
\eann
for some positive constants $\alpha_i$ and $\beta_i$, which depend on $\varepsilon$  but not on $\sigma$ or $p$, and $K_{\varepsilon,\sigma,p}$, which depends on $\varepsilon$, $\sigma$ and $p$. 

It is clear that $L>0$ and $\ell>0$ may be chosen such that 
\bann
\left(\alpha_0p^2-\alpha_1\sigma p^{\frac{5}{2}}-\alpha_2 p^{\frac{3}{2}}-\alpha_3 p\right)\geq 0
\eann
and
\bann
\left(\beta_0 p-\beta_1\sigma p^{\frac{3}{2}}-\beta_2\sigma p-\beta_3p^{\frac{1}{2}} \right) \geq 0
\eann
for all $p>L$ and $0<\sigma<\ell p^{-\frac{1}{2}}$. The claim then follows by integrating with respect to the time variable.
\end{proof}

The proof of Theorem \ref{thm:CE} is completed by proceeding with Huisken's Stampacchia-De Giorgi iteration scheme. We omit these details as the arguments required already appear in \cite[Section 5]{hypersurfaces} with no significant changes necessary.

\end{document}